\documentclass{article}
\usepackage{amssymb, amsmath, amsthm}

\DeclareMathOperator{\GL}{GL}
\DeclareMathOperator{\SL}{SL}
\DeclareMathOperator{\diam}{diam}
\DeclareMathOperator{\sgn}{sgn}
\DeclareMathOperator{\tr}{tr}
\DeclareMathOperator{\centr}{\mathcal{Z}}
\DeclareMathOperator{\centraliser}{\mathcal{C}}
\DeclareMathOperator{\M}{M}
\newcommand{\Z}{\mathbb{Z}}
\newcommand{\A}{s}
\newcommand{\B}{t}
\newcommand{\Cg}{\Gamma}
\newcommand{\C}{\centraliser}

\DeclareMathOperator{\PSL}{PSL}

\newtheorem{theorem}{Theorem}[section]
\newtheorem{lemma}[theorem]{Lemma}
\newtheorem{corollary}[theorem]{Corollary}

\title{The diameters of commuting graphs of linear groups and matrix rings over the integers modulo~$m$}
\author{Michael Giudici and Aedan Pope\footnote{This work was completed while the second author was an Australian Mathematical Sciences Institute Vacation Scholar. The first author holds an Australian Research Fellowship.} \\
        School of Mathematics and Statistics\\ 
        The University of Western Australia\\ 
        35 Stirling Highway\\
		Crawley WA 6009\\
		giudici@maths.uwa.edu.au, popea02@student.uwa.edu.au}
\date{}

\begin{document}

\maketitle

\begin{abstract}
The commuting graph of a group $G$, denoted by $\Cg(G)$, is the simple undirected graph whose vertices are the non-central elements of $G$ and two distinct vertices are adjacent if and only if they commute. Let $\Z_m$ be the commutative ring of equivalence classes of integers modulo $m$. In this paper we investigate the connectivity and diameters of the commuting graphs of $\GL(n,\Z_m)$ to contribute to the conjecture that there is a universal upper bound on $\diam(\Cg(G))$ for any finite group $G$ when $\Cg(G)$ is connected. For any composite $m$, it is shown that $\Cg(\GL(n,\Z_m))$ and $\Cg(\M(n,\Z_m))$ are connected and $\diam(\Cg(\GL(n,\Z_m)))=\diam(\Cg(\M(n,\Z_m)))=3$. For $m$ a prime, the instances of connectedness and absolute bounds on the diameters of $\Cg(\GL(n,\Z_m))$ and $\Cg(\M(n,\Z_m))$ when they are connected are concluded from previous results.
\end{abstract}

\section{Introduction}

For a group $G$, we denote the \emph{center} of $G$ by $\centr(G)$ and $\centr(G)=\{x\in G | xy=yx~\forall y \in G \}$. If $x$ is an element of $G$, then $\C_G(x)$ denotes the \emph{centraliser} of $x$ in $G$ and $\C_G(x)=\{y\in G|xy=yx\}$. The \emph{commuting graph} of a group, denoted by $\Cg(G)$, is the simple undirected graph whose vertices are the non-central elements of $G$ and two distinct vertices $x$ and $y$ are adjacent if and only if $xy=yx$. We take analogous definitons for the center, centraliser and commuting graph of a ring $R$. A \emph{path} in a graph is an ordered list $a_1, a_2, \ldots , a_k$ of vertices where there is an edge in the graph from $a_i$ to $a_{i+1}$ for all $i$; the path is said to between $a_1$ and $a_k$ and of length $k-1$. A graph is \emph{connected} if and only if there exists a path between any two distinct vertices in the graph. The \emph{distance} between two vertices of a graph, say $x$ and $y$, is the length of the shortest path between $x$ and $y$ in the graph if such a path exists and is $\infty$ otherwise; this is denoted $\operatorname{d}(x,y)$. The \emph{diameter} of a graph $\Gamma$ is the maximum distance between any two vertices in the graph, and is denoted $\diam(\Gamma)=\max\{\operatorname{d}(x,y)| x,y \in \Gamma\}$. We use $\M(n,R)$ to denote the ring of all $n \times n$ matrices over the ring $R$, $\GL(n,R)$ to denote the group all invertible $n \times n$ matrices over $R$ and $\SL(n,R)$ to denote those with determinant 1. $\Z_m$ denotes the commutative ring of equivalence classes of integers modulo $m$. 

The commuting graphs of groups have been studied heavily, for example in \cite{brauer, iranmanesh, segev, segev2}, and those of rings in \cite{akbari2, akbari3, akbari}. In \cite{iranmanesh}, Iranmanesh and Jafarzadeh conjecture that there is a universal upper bound on the diameter of a connected commuting graph for any finite nonabelian group. They then determine when the commuting graph of a symmetric or alternating group is connected and that the diameter is at most 5 in this case. The paper \cite{segev} proves that for all finite classical simple groups over a field of size at least 5, when the commuting graph of a group is connected then its diameter is at most 10. Previous research into the diameters of the commuting graphs of linear groups and matrix rings has primarily covered these over fields. For a field $F$, the authors of \cite{akbari} show that $\diam(\Cg(\GL(n,F))) \leq \diam(\Cg(\M(n,F))) \leq 6$ when these graphs are connected and $|F|$ is greater than or equal to $3$. In addition, \cite{akbari2} provides necessary and sufficient conditions for $\Cg(\SL(n,F))$ to be connected; an upper bound on the diameter of $\Cg(\SL(n,F))$ can be generated from the proof. Our paper adds to this body of evidence supporting the conjecture by calculating the diameter of the commuting graphs of some general linear groups over commutative rings that are not fields. The diameters of the corresponding matrix rings are also calculated. 

\begin{theorem}
\label{composite}
Let $m$ be a composite natural number and $n \geq 2$. Then $\Cg(\GL(n,\Z_m))$ and $\Cg(\M(n,\Z_m))$ are connected and $\diam(\Cg(\GL(n,\Z_m)))=\diam(\Cg(\M(n,\Z_m)))=3$.
\end{theorem}

Consider when $m$ is a prime. If $m \geq 3$ and $n \geq 2$, by \cite[Corollaries 7 and 11]{akbari2} and \cite[Theorems 14 and 17]{akbari}, $\Cg(\GL(n,\Z_m))$ and $\Cg(\M(n,\Z_m))$ are connected if and only if $n$ is not prime and $4 \leq \diam(\Cg(\GL(n,\Z_m))) \leq \Cg(\M(n,\Z_m)) \leq 6$ in this case. In the case of $m=2$ and $n \geq 2$, \cite[Corollary 7]{akbari2} shows that $\Cg(\M(n,\Z_2))$ is connected if and only if $n$ is not a prime number and \cite[Corollary 14]{akbari2} gives that $\Cg(\GL(n,\Z_2))$ is connected if and only if $n$ and $n-1$ are not prime numbers. Moreover, $\diam(\Cg(\M(n,\Z_2))) \leq 6$ when it is connected \cite[Theorem 17]{akbari}. By modifying the proof of \cite[Theorem 12.5]{segev}, one can conclude that, for $n \geq 5$, an arbitrary element of $\GL(n,\Z_2)=\PSL(n,\Z_2)$ is distance at most 3 from a transvection of $\GL(n,\Z_2)$ and that the distance between any two transvections is at most 2. Thus for $n\geq 5, \diam(\Cg(\GL(n,\Z_2))) \leq 8$; if $n < 5$ then $n$ or $n-1$ is prime and the graph is disconnected from above. Therefore, for any integers $m$ and $n$ that are greater than $1$, when the corresponding commuting graphs are connected there is a universal upper bound on $\diam(\Cg(\GL(n,\Z_m)))$ and $\Cg(\M(n,\Z_m))$.

We record here that we have used the \textsc{Magma}\cite{magma} implementation of the small group database \cite{smallgroups, millennium} to calculate the connectedness and diameters of the commuting graphs for all groups of order at most 2000, except for the orders of the form $k2^6$ for $k\neq 9, 4$ composite. For connected graphs, the largest diameter found was $6$.

We do not know if the bound of 8 for the diameter of $\GL(n,2)$ is sharp. No example of a connected commuting graph of a group with diameter greater than 6 was found in previous literature and so it would be interesting to find examples of diameter greater than 6.

\section{Results}
We begin with the following result pertaining to commutative graphs which is a generalisation of the disconnectedness of $\Cg(\M(2,F))$ for $F$ a field, concluded in \cite[Remark 9]{akbari3}

\begin{theorem}
\label{integraldomain}
If R is an integral domain, then $\Cg(\M(2,R))$, $\Cg(\GL(2,R))$ and $\Cg(\SL(2,R))$ are disconnected.
\end{theorem}
\begin{proof}
Let $R$ be an integral domain.\\
Let $\mathcal{A}=\left\{ \begin{array}{c|c} \begin{bmatrix} a_1 & a_2 \\ 0 & a_1 \end{bmatrix} \in \M(2,R) & 
a_1, a_2 \in R, a_2 \neq 0 \end{array} \right\} \subseteq \M(2,R) \setminus \centr(\M(2,R))$.
Let $A =\begin{bmatrix} a_1 & a_2 \\ 0 & a_1 \end{bmatrix}\in \mathcal{A}$  for some $a_i \in R$ and let $X = \begin{bmatrix} x_1 & x_2 \\ x_3 & x_4 \end{bmatrix} \in \C_{\M(2,R)}(A)\setminus \centr(\M(2,R))$. Then
\begin{align*}
XA&=\begin{bmatrix} x_1 & x_1 \\ x_3 & x_4 \end{bmatrix}\begin{bmatrix} a_1 & a_2 \\ 0 & a_1 \end{bmatrix}
=\begin{bmatrix} a_1x_1 & a_2x_1+a_1x_2 \\ a_1x_3 & a_2x_3+a_1x_4 \end{bmatrix}\\
AX&=\begin{bmatrix} a_1 & a_2 \\ 0 & a_1 \end{bmatrix}\begin{bmatrix} x_1 & x_1 \\ x_3 & x_4 \end{bmatrix}
=\begin{bmatrix} a_1x_1+a_2x_3 & a_1x_2+a_2x_4 \\ a_1x_3 & a_1x_4 \end{bmatrix}
\end{align*}
Since $XA=AX$, this yields $a_1x_1=a_1x_1+a_2x_3$ and $a_2x_1+a_1x_2=a_1x_2+a_2x_4$. Now $a_1x_1=a_1x_1+a_2x_3$ implies $a_2x_3 = 0$. Since $a_2 \neq 0$ and there are no zero divisors in $R$, $x_3 = 0$. Moreover, $a_2x_1+a_1x_2=a_1x_2+a_2x_4$ implies $a_2x_1=a_2x_4$ which yields $x_1=x_4$ by cancelling $a_2$ in the integral domain. Thus $X = \begin{bmatrix} x_1 & x_2 \\ 0 & x_1 \end{bmatrix}$. As $X \notin \centr(\M(2,R))$, it must not be a scalar matrix and so $x_2 \neq 0$. Therefore $X \in \mathcal{A}$. So in $\Cg(\M(2,R))$, $\mathcal{A}$ forms an isolated connected component. Thus the matrices $B=\begin{bmatrix} 1 & 1 \\ 0 & 1 \end{bmatrix} \in \mathcal{A} \cap \SL(2,R)$ and $C=\begin{bmatrix} 1 & 0 \\ 1 & 1 \end{bmatrix} \in \SL(2,R) \setminus \centr(\SL(2,R))$ with $C \notin \mathcal{A}$ are in different connected components of $\Cg(\M(2,R))$, $\Cg(\GL(2,R))$ and $\Cg(\SL(2,R))$. Therefore the graphs $\Cg(\M(2,R))$, $\Cg(\GL(2,R))$ and $\Cg(\SL(2,R))$ are disconnected.
\end{proof}

The following are some useful lemmas on the properties of the ring of integers modulo $m$.

\begin{lemma}
\label{units}
Let u,v,s,t be pairwise coprime integers. Then for any natural numbers k,l, the two integers $us^k+vt^l$ and st are coprime.
\end{lemma}
\begin{proof}
Let $u,v,s,t$ be pairwise coprime integers and $k, l$ be natural numbers. Let $d=\gcd(us^k+vt^l, st)$. Assume that $d \neq 1$ and let $p$ be a prime that divides $d$. Then $p|st$ and since $p$ is prime, $p|s$ or $p|t$. Without loss of generality, assume $p|s$, and thus $p|us^k$. As $v,s,t$ are pairwise coprime, $p|s$ implies that $p\nmid v$ and $p\nmid  t$, thus $p \nmid vt^m$. So $p \nmid us^k+vt^k$, a contradiction. Therefore $d=1$ and $us^k+vt^l, st$ are coprime.
\end{proof}

\begin{lemma}
\label{dets}
For coprime natural numbers $s$ and $t$ greater than 1, if $X$ and $Y$ are elements of $\GL(n,\Z_{st})$ then $sX+tY$ is also an element of $\GL(n,\Z_{st})$.
\end{lemma}
\begin{proof}

Let $s, t$ be coprime natural numbers greater than 1 and $X,Y \in \GL\left(n,\Z_{st}\right)$. Then by the Leibniz formula for the determinant of an $n\times n$ matrix,

\begin{align*} 
\det(sX+tY) &= \sum_{\sigma \in S_n} \left(\sgn(\sigma)\prod_{i=1}^{n} (sX+tY)_{i,\sigma(i)}\right)\\
&=\sum_{\sigma \in S_n} \sgn(\sigma)(sX_{1,\sigma(1)}+tY_{1,\sigma(1)})(sX_{2,\sigma(2)}+tY_{2,\sigma(2)}) \cdots (sX_{n,\sigma(n)}+tY_{n,\sigma(n)})\\
&=\sum_{\sigma \in S_n} \sgn(\sigma)((sX_{1,\sigma(1)}sX_{2,\sigma(2)}\cdots sX_{n,\sigma(n)}) + (tY_{1,\sigma(1)}tY_{2,\sigma(2)}\cdots tY_{n,\sigma(n)})) \\&\text{(expanding the product: every term that is }s^at^bh\text{ for }a,b \geq 1\text{ and some }h\\&\text{ a product of entries of }X\text{ and }Y\text{ is }0\text{ as }s^at^b =0\text{)}\\
&=\sum_{\sigma \in S_n} \left(\sgn(\sigma)\prod_{i=1}^{n} sX_{i,\sigma(i)}+ \sgn(\sigma)\prod_{i=1}^{n} tY_{i,\sigma(i)}\right)\\
&=\sum_{\sigma \in S_n}\left(\sgn(\sigma)\prod_{i=1}^{n} sX_{i,\sigma(i)}\right)+ \sum_{\sigma \in S_n} \left(\sgn(\sigma)\prod_{i=1}^{n} tY_{i,\sigma(i)}\right)\\
&=\det(sX)+\det(tY)\\
&=s^n\det(X) + t^n\det(Y)
\end{align*}
Since $\det(X)$ and $\det(Y)$ are units in $\Z_{st}$, they are coprime to $s$ and $t$ and so, by Lemma \ref{units}, $s^n\det(X) + t^n\det(Y)$ is coprime to $st$ and hence is a unit in $\Z_{st}$. Therefore $sX+tY$ is invertible and is an element of $\GL(n,\Z_{st})$.
\end{proof}

For the remainder of this paper we use $I, 0, I_{r}$ and $0_{r\times s}$ to denote the identity, zero, $r \times r$ identity and $r \times s$ zero matrices respectively. We also use $E_{i,j}$ to denote the matrix with $(i,j)$ entry equal to 1 and every other entry equal to 0. We now obtain a lower bound of $3$ on the diameters of $\Cg(\GL(n,\Z_m))$ and $\Cg(\M(n,\Z_m))$ for an arbitrary $m$ and $n$.

\begin{lemma}
For any $n,m \geq 2$, the matrix $P=I_{n} + \begin{bmatrix} 0_{n-1 \times 1} & I_{n-1} \\ 0 & 0_{1 \times n-1} \end{bmatrix} \in \GL(n,\Z_m)\setminus\centr(\GL(n,\Z_m))$ has the property that $\C_{\M(n,\Z_m)}(P) \cap \C_{\M(n,\Z_m)}(P^T)=\centr(\M(n,\Z_m))$.
\end{lemma}
\begin{proof}

Let $P$ be the matrix described in the hypothesis.
Then 
$$P =
\begin{bmatrix}
1 & 1 & 0 & \ldots &  0 \\
0 & 1 & 1 & \ldots & 0 \\
\vdots & \ddots & \ddots &\ddots& \vdots \\
0 & \ldots & 0 & 1 & 1 \\
0 & \ldots & 0 & 0 & 1
\end{bmatrix} \text{ and }P^T =
\begin{bmatrix}
1 & 0 & \ldots & 0 & 0\\
1 & 1 & \ddots & \vdots & \vdots \\
0 & 1 & \ddots & 0 & 0 \\
\vdots & \vdots & \ddots & 1 & 0 \\
0 & 0 & \ldots & 1 & 1 \\
\end{bmatrix}$$
Let $X \in \C_{\M(n,\Z_m)}(P)$. So $X = \begin{bmatrix}
x_{1,1} & x_{1,2} & \ldots & x_{1,n} \\
x_{2,1} & x_{2,2} & \ldots & x_{2,n} \\
\vdots & \vdots & \ddots & \vdots \\
x_{n,1} & x_{n,2} & \ldots & x_{n,n} \\
\end{bmatrix}$ for some $x_{i,j} \in \Z_{m}$.
\begin{align*}
\text{Then }PX &=
\begin{bmatrix}
x_{1,1}+x_{2,1} & x_{1,2}+x_{2,2} & \ldots & x_{1,n}+x_{2,n} \\
x_{2,1}+x_{3,1} & x_{2,2}+x_{3,2} & \ldots & x_{2,n}+x_{3,n} \\
\vdots & \vdots & \ddots & \vdots \\
x_{n-1,1}+x_{n,1} & x_{n-1,2}+x_{n,2} & \ldots & x_{n-1,1n}+x_{n,n} \\
x_{n,1} & x_{n,2} & \ldots & x_{n,n}
\end{bmatrix} \\
\text{and }XP &=
\begin{bmatrix}
x_{1,1} & x_{1,1}+x_{1,2} & \ldots & x_{1,n-1}+x_{1,n} \\
x_{2,1} & x_{2,1}+x_{2,2} & \ldots & x_{2,n-1}+x_{2,n} \\
\vdots & \vdots & \ddots & \vdots \\
x_{n,1} & x_{n,1}+x_{n,2} & \ldots & x_{n,n-1}+x_{n,n}
\end{bmatrix}.
\end{align*}
Since $PX=XP$, we obtain
\[ \begin{array}{ccccc}
x_{1,1}+x_{2,1} = x_{1,1} & 
x_{1,2}+x_{2,2} = x_{1,1}+x_{1,2} &
\ldots &
x_{1,n}+x_{2,n} = x_{1,n-1}+x_{1,n} \\
x_{2,1}+x_{3,1} = x_{2,1} & 
x_{2,2}+x_{3,2} = x_{2,1}+x_{2,2} &
 \ldots & 
 x_{2,n}+x_{3,n} = x_{2,n-1}+x_{2,n}\\
\vdots & \vdots & \ddots & \vdots \\
x_{n,1} = x_{n,1} & 
x_{n,2} = x_{n,1}+x_{n,2} & 
\ldots & 
x_{n,n} = x_{n,n-1}+x_{n,n} \\
\end{array} \]
The left column of equations from $PX=XP$ gives
\begin{align*}
x_{1,1} +x_{2,1}&= x_{1,1}\text{, so } x_{2,1}=0.\\
x_{2,1} +x_{3,1}&= x_{2,1}\text{, so } x_{3,1}=0.\\
&\vdots\\
x_{n-1,1} +x_{n,1}&= x_{n-1,1}\text{, so } x_{n,1}=0.
\end{align*}
The second column gives
\begin{align*}
x_{1,2} +x_{2,2}&= x_{1,1} + x_{1,2}\text{, so }x_{2,2}=x_{1,1}.\\
x_{2,2} +x_{3,2}&= x_{2,1} + x_{2,2}\text{, from above } x_{2,1}=0\text{ and so }x_{3,2}=0.\\
x_{3,2} +x_{4,2}&= x_{3,1} + x_{3,2}\text{, from above } x_{3,1}=0\text{ and so }x_{4,2}=0.\\
&\vdots\\
x_{n-1,2} +x_{n,2}&= x_{n-1,1}+x_{n-1,2}\text{, from above } x_{n-1,1}=0\text{ and so }x_{n,2}=0.\\
\end{align*}
The third column gives $x_{3,3}=x_{2,2}$ and then $x_{k,3} = 0$ for all $k \geq 4$. This continues across the columns and thus $x_{i,i}=x_{1,1}$ for all $i$ and $x_{j,k}=0$ whenever $j > k$. So $X$ has the form
\[
X=\begin{bmatrix}
x_{1,1} & x_{1,2} & \ldots & x_{1,n} 	\\
0 	& x_{1,1}  & \ldots & x_{2,n} 	\\
\vdots 	& \vdots  & \ddots    & \vdots  	\\
0 	& 0 	        & \ldots  & x_{1,1}
\end{bmatrix}
\]

Let $Y \in \C_{\M(n,\Z_m)}(P^T)$ where $Y_{i,j}=y_{i,j}$ for some $y_{i,j}\in \Z_{m}$. By similar arithmetic, $Y$ must have the form
\[
Y=\begin{bmatrix}
y_{1,1} 	& 0 	  		& \ldots 	& 0 		\\
y_{2,1} 	& y_{1,1} 			& \ldots 	& 0 		\\
\vdots 		& \vdots  		& \ddots 	& \vdots	\\
y_{n,1} 	& y_{n,2} 	& \ldots 	& y_{1,1}
\end{bmatrix}
\]
Thus 
\[
\C_{\M(n,\Z_m)}(P) \cap \C_{\M(n,\Z_m)}(P^T)=
\left\{ \begin{array}{c|c} 
\begin{bmatrix} 
s 	& 0 	 & \ldots & 0 		\\
0 	& s  	 & \ldots & 0 		\\
\vdots 	& \vdots & \ddots & \vdots 	\\
0 	& 0	 & \ldots & s
\end{bmatrix}
& s \in \Z_{m}
\end{array} \right\} = \centr(\M(n,\Z_m)).
\qedhere\]
\end{proof}

\begin{corollary}
\label{lowerbound}
For any $m, n \geq 2$, $\diam(\Cg(\GL(n,\Z_m)))$ and $\diam(\Cg(\M(n,\Z_m)))$ are at least $3$.
\end{corollary}

The following lemmas discern some properties of $\Cg(\GL(n,\Z_m))$ and $\Cg(\M(n,\Z_m))$ when $m$ is a prime power.

\begin{lemma}
\label{aprimepower}
If $p$ is prime and $t \geq 2$ is a natural number then for any $X \in \M(n,\Z_{p^t})$ there exists $Y \in \M(n,\Z_{p^t})$ such that $X$ commutes with $p^{t-1} Y + I$ and $p^{t-1} Y + I \in \GL(n,\Z_{p^t}) \setminus \centr(\M(n,\Z_{p^t})$.
\end{lemma}
\begin{proof}
Let $p$ be a prime and $t \geq 2$ a natural number. Let $X \in \M(n,\Z_{p^t})$. To find $Y$ we will divide into 2 cases. Firstly, for all $i \neq j$ there exists $u_{i,j} \in \Z_{p^t}$ such that $X_{i,j}=pu_{i,j}$. Secondly, there exist distinct $v, w$ such that $X_{v,w} \neq pu$ for any $u \in \Z_{p^t}$.

Suppose that we have the first case. Then $X$ commutes with $p^{t-1}E_{1,1} + I \in \M(n,\Z_{p^t}) \setminus \centr(\M(n,\Z_{p^t}))$. Further, $(p^{t-1}E_{1,1} + I)((-p^{t-1})E_{1,1} + I)=p^{t-1}(-p^{t-1})E_{1,1}+I=I$, so $p^{t-1}E_{1,1} + I \in \GL(n,\Z_{p^t}) \setminus \centr(\M(n,\Z_{p^t})$ and this case is done.

Now suppose that there exist distinct $v, w$ such that $X_{v,w} \neq pu$ for any $u \in \Z_{p^t}$. Let $A = p^{t-1} X + I\in \M(n,\Z_{p^t})$. Since $p$ is not a factor of $X_{v,w}$, $A_{v,w}=p^{t-1}X_{v,w} \neq 0$ and $v \neq w$ gives that $A$ is not diagonal and thus not a scalar. Now, $A$ clearly commutes with $X$ and $A(-p^{t-1} X + I)=I$ so $A=p^{t-1} X + I \in \GL(n,\Z_{p^t}) \setminus \centr(\M(n,\Z_{p^t})$, this case is also done.
\end{proof}

\begin{lemma}
\label{primepower}
If $p$ is prime and $t \geq 2$ is a natural number, then $\Cg(\M(n,\Z_{p^t}))$ and $\Cg(\GL(n,\Z_{p^t}))$ are connected and $\diam(\Cg(\M(n,\Z_{p^t})))=\diam(\Cg(\GL(n,\Z_{p^t})))=3$.
\end{lemma}
\begin{proof}
Let $p$ be a prime and $t \geq 2$ a natural number. Let $X, Y$ be arbitrary vertices in $\Cg(\M(n,\Z_{p^t}))$. By Lemma \ref{aprimepower}, there exist $A, B \in \M(n,\Z_{p^t})$ such that $p^{t-1} A + I \in \C_{\M(n,\Z_{p^t})}(X) \cap \GL(n,\Z_{p^t}) \setminus \centr(\M(n,\Z_{p^t}))$ and $p^{t-1} B + I \in \C_{\M(n,\Z_{p^t})}(Y) \cap \GL(n,\Z_{p^t}) \setminus \centr(\M(n,\Z_{p^t}))$. Now, 
$(p^{t-1} A + I)(p^{t-1} B + I)=
p^{t-1}p^{t-1}AB+p^{t-1} AI+p^{t-1} IB+ I=
p^{t}p^{t-2}AB+p^{t-1} AI+p^{t-1} IB+ I=
0+ p^{t-1}IA+p^{t-1}BI+I=
p^{t-1}p^{t-1}BA+p^{t-1}IA+p^{t-1}BI+I=
(p^{t-1} B + I)(p^{t-1} A + I)$. So $X\sim p^{t-1} A + I\sim p^{t-1} B + I\sim Y$ is a path of invertible matrices of length $3$ between $X$ and $Y$ in $\Cg(\M(n,\Z_{p^t}))$, if $X, Y \in \GL(n,\Z_{p^t})$ then this path is also in $\Cg(\GL(n,\Z_{p^t}))$. Therefore $\Cg(\M(n,\Z_{p^t}))$ and $\Cg(\GL(n,\Z_{p^t}))$ are connected and,  from the lower bound given by Corollary \ref{lowerbound}, $\diam(\Cg(\M(n,\Z_{p^t})))=\diam(\Cg(\GL(n,\Z_{p^t})))=3$.
\end{proof}

Now we obtain some lemmas on the nature of $\Cg(\GL(n,\Z_m))$ and $\Cg(\M(n,\Z_m))$ when $m$ is the product of two coprime factors.

\begin{lemma}
\label{amulticomp}
If $\A, \B$ are coprime natural numbers greater than 1 then for any $X \in \M(n,\Z_{\A\B})$ there exist $Y \in \M(n,\Z_{\A\B})$ and $k \in \Z_{\A\B}$ such that $\A Y + kI \in \C_{\M(n,\Z_{\A\B})}(X) \setminus \centr(\M(n,\Z_{\A\B}))$. Moreover, if $X \in \GL(n,\Z_{\A\B})$, then $Y$ and $k$ can be chosen so that $\A Y + kI \in \GL(n,\Z_{\A\B})$.
\end{lemma}
\begin{proof}
Let $\A, \B$ be coprime natural numbers greater than 1. Let $X \in \M(n,\Z_{\A\B})$. Firstly assume that for all $i \neq j$ there exists $u_{i,j} \in \Z_{\A\B}$ such that $X_{i,j}=\B u_{i,j}$. Then $X$ commutes with $\A E_{1,1} + I \in \M(n,\Z_{\A\B}) \setminus \centr(\M(n,\Z_{\A\B}))$. 
Now assume there exist distinct $v, w$ such that $X_{v,w} \neq \B u$ for any $u \in \Z_{\A\B}$. Then $X$ commutes with $\A X + I \in \M(n,\Z_{\A\B})$. Moreover, $X_{v,w}$ is not a multiple of $\B$ and so when multiplied by $\A$ does not give zero. Since $v \neq w$, this gives that $\A X + I$ is not diagonal and thus nonscalar. This covers the first part of the lemma.

Now let $X$ be invertible. To find appropriate $Y$ and $k$, we will divide into 3 cases: (1) $\B = 2$ and for all $i \neq j$ there exists $u_{i,j} \in \Z_{\A\B}$ such that $X_{i,j}=\B u_{i,j}$. (2) $\B \neq 2 $ and for all $i \neq j$ there exists $u_{i,j} \in \Z_{\A\B}$ such that $X_{i,j}=\B u_{i,j}$. (3) There exist distinct $v, w$ such that $X_{v,w} \neq \B u$ for any $u \in \Z_{\A\B}$.

Case 1: Here $\B=2$ and $X$ is of the form
$$X = \begin{bmatrix}
x_{1,1} & \B u_{1,2} & \ldots & \B u_{1,n} \\
\B u_{2,1} & x_{2,2} & \ldots & \B u_{2,n} \\
\vdots & \vdots & \ddots & \vdots \\
\B u_{n,1} & \B u_{n,2} & \ldots & x_{n,n} \\
\end{bmatrix}\text{ for some }x_{i,i}, u_{i,j} \in \Z_{\A\B}.$$
Consider that $\det(X)$ is a sum of multiples of permutations of $n$ entries of $X$ with precisely one entry from each row and column. All of the terms in the summation will have a factor of $\B$ in them except for the $x_{1,1}x_{2,2}\cdots x_{n,n}=\tr(X)$ term. If one of $x_{i,i}$ were a multiple of $\B$ then so would this term, and $\det(X)=z\B$ for some $z$. This is not a unit in $\Z_{\A\B}$ so any such $X$ is not invertible. Therefore, all $x_{i,i}$ coprime to $\B$, that is, are odd. As $\B=2$, $z\A \equiv 0\pmod{\A\B}$ if $z$ is even and $z\A\equiv\A\pmod{\A\B}$ if $z$ is odd. So $\A x_{i,i}=\A$ for all $i$. Let $A = \A (E_{1,n}+I) + \B I \in \M(n,\Z_{\A\B})$. Then $A$ is invertible by Lemma \ref{dets}. Since $A$ is also not a scalar, $A \in \GL(n,\Z_{\A\B})\setminus \centr(\GL(n,\Z_{\A\B}))$. Now \begin{align*}
AX
=&
 \begin{bmatrix}
(\A+\B) x_{1,1} +\A\B u_{n,1} &  (\A+\B) \B u_{1,2} +\A\B u_{n,2}& \ldots &(\A+\B) \B u_{1,n}+\A x_{n,n} \\
(\A+\B)\B u_{2,1} & (\A+\B)x_{2,2} & \ldots &  (\A+\B)\B u_{2,n} \\
\vdots & \vdots & \ddots  & \vdots \\
(\A+\B)\B u_{n,1} & (\A+\B)\B u_{n,2} &  \ldots &(\A+\B)x_{n,n}
\end{bmatrix}\\
=&
\begin{bmatrix}
(\A+\B) x_{1,1} &  (\A+\B)\B u_{1,2} &   \ldots &   (\A+\B) \B u_{1,n} + \A \\
(\A+\B)\B u_{2,1} & (\A+\B)x_{2,2} &  \ldots &   (\A+\B)\B u_{2,n} \\
\vdots & \vdots &  \ddots & \vdots  \\
(\A+\B)\B u_{n,1} & (\A+\B)\B u_{n,2} & \ldots & (\A+\B)x_{n,n}
\end{bmatrix}\\
=&
\begin{bmatrix}
(\A+\B) x_{1,1} &  (\A+\B)\B u_{1,2} &   \ldots &   (\A+\B) \B u_{1,n}+ \A x_{1,1} \\
(\A+\B)\B u_{2,1} & (\A+\B)x_{2,2} &  \ldots &   (\A+\B)\B u_{2,n} + \A\B u_{2,1} \\
\vdots & \vdots &   \ddots & \vdots  \\
(\A+\B)\B u_{n,1} & (\A+\B)\B u_{n,2} & \ldots & (\A+\B)x_{n,n}+ \A\B u_{n,1}
\end{bmatrix}\\
=&XA
\end{align*} So $A \in \C_{\GL(n,\Z_{\A\B})}(X) \setminus \centr(\GL(n,\Z_{\A\B}))$ and this case is done.

Case 2: Here $\B \neq 2$ and for all $i \neq j$ there exists $u_{i,j} \in \Z_{\A\B}$ such that $X_{i,j}=\B u_{i,j}$. Let $A=\A(I-2E_{1,1}) + \B I \in \M(n,\Z_{\A\B})$. Then $A$ is invertible by Lemma \ref{dets}. Since $\B \neq 2$, we have $\A \neq -\A$ and so $A_{1,1}=-\A+\B\neq \A+\B=A_{2,2}$. Thus $A$ is not a scalar. Writing $X_{i,i}=x_{i,i}$ and $X_{i,j}=t u_{i,j}$ for $i \neq j$ as in the previous case, 
\begin{align*}
AX
=&
 \begin{bmatrix}
(-\A+\B) x_{1,1} &  (-\A+\B) \B u_{1,2} &   \ldots &  (-\A+\B) \B u_{1,n} \\
(\A+\B)\B u_{2,1} & (\A+\B)x_{2,2} &  \ldots &  (\A+\B)\B u_{2,n} \\
\vdots & \vdots &   \ddots & \vdots \\
(\A+\B)\B u_{n,1} & (\A+\B)\B u_{n,2} &   \ldots & (\A+\B)x_{n,n}
\end{bmatrix}\\
=&
\begin{bmatrix}
(-\A+\B) x_{1,1} &  \B^2 u_{1,2} &  \ldots &  \B^2 u_{1,n} \\
\B^2 u_{2,1} & (\A+\B)x_{2,2} &  \ldots &  \B^2 u_{2,n} \\
\vdots & \vdots &  \ddots & \vdots \\
\B^2 u_{n,1} & \B^2 u_{n,2} &  \ldots & (\A+\B)x_{n,n}
\end{bmatrix}\\
=&
\begin{bmatrix}
(-\A+\B) x_{1,1} &  (\A+\B) \B u_{1,2} &  \ldots &  (\A+\B) \B u_{1,n} \\
(-\A+\B)\B u_{2,1} & (\A+\B)x_{2,2} &  \ldots &  (\A+\B)\B u_{2,n} \\
\vdots & \vdots &   \ddots & \vdots \\
(-\A+\B)\B u_{n,1} & (\A+\B)\B u_{n,2} &  \ldots & (\A+\B)x_{n,n}
\end{bmatrix}\\
=&XA
\end{align*}So $A \in \C_{\GL(n,\Z_{\A\B})}(X) \setminus \centr(\GL(n,\Z_{\A\B}))$ and this case is done.

Case 3: Now there exist distinct $v, w$ such that $X_{v,w} \neq \B u$ for any $u \in \Z_{\A\B}$. Let $A = \A X + \B I \in \M(n,\Z_{\A\B})$, which is invertible by Lemma \ref{dets}. Now $A_{v,w}=\A X_{v,w} \neq 0$ as $X_{v,w}$ is not a multiple of $\B$. Since $v \neq w$, $A$ is nonscalar. Then $A$ clearly commutes with $X$ and so $A \in \C_{\GL(n,\Z_{\A\B})}(X) \setminus \centr(\GL(n,\Z_{\A\B}))$ and the lemma is proved.
\end{proof}

\begin{lemma}
\label{multicomp}
If $\A, \B$ are coprime natural numbers greater than $1$, then $\Cg(\GL(n,\Z_{\A\B}))$ and $\Cg(\M(n,\Z_{\A\B}))$ are connected and $\diam(\Cg(\GL(n,\Z_{\A\B})))=\diam(\Cg(\M(n,\Z_{\A\B})))=3$.
\end{lemma}
\begin{proof}
Let $\A, \B$ be coprime natural numbers greater than 1. Let $X, Y$ be arbitrary vertices in $\Cg(\M(n,\Z_{m}))$. By Lemma \ref{amulticomp}, there exist $A, B \in \M(n,\Z_{m})$ and $k, \ell \in \Z_{m}$ such that $\A A + kI \in \C_{\M(n,\Z_{m})}(X) \setminus \centr(\M(n,\Z_{m}))$ and $\B B + \ell I \in \C_{\M(n,\Z_{m})}(Y) \setminus \centr(\M(n,\Z_{m}))$. Now, 
$(\A A + kI)(\B B + \ell I)=
\A\B AB+\A \ell AI+k\B IB+k\ell I=
0+\ell\A IA+\B kBI +\ell kI=
\B\A BA+\ell\A IA+\B kBI +\ell kI=
(\B B + \ell I)(\A A + kI)$. So $X\sim \A A + kI\sim \B B + \ell I\sim Y$ is a path of length $3$ between $X$ and $Y$. Therefore $\Cg(\M(n,\Z_{m}))$ is connected and, with Corollary \ref{lowerbound}, $\diam(\Cg(\M(n,\Z_{m})))=3$

Moreover,  if $X, Y \in \Cg(\GL(n,\Z_{m}))$ then, by Lemma \ref{amulticomp}, $A, B, k$ and $\ell$ can be chosen such that $\A A + kI A, \B B + \ell I B \in \GL(n,\Z_{m})$. Thus intermediate vertices on the path given above can be replaced with invertible ones, and so $\Cg(\GL(n,\Z_{m}))$ is connected and, using Corollary \ref{lowerbound}, $\diam(\Cg(\GL(n,\Z_{m})))=3$.
\end{proof}

Since any composite integer is either a prime power or a product of two coprime factors, the proof of Theorem~\ref{composite} is concluded by combining Lemmas~\ref{primepower} and~\ref{multicomp}.

\end{document}